\documentclass[10pt]{article}

\usepackage{a4wide}

\usepackage{amsmath,amsbsy,amssymb,latexsym}
\usepackage{amsthm}
\usepackage{url}

\newtheorem{thm}{Theorem}

\newtheorem{lem}{Lemma}

\theoremstyle{remark}
\newtheorem{rem}{Remark}

\theoremstyle{definition}
\newtheorem{defn}{Definition}

\newtheorem{ex}{Example}

 \newcommand{\rl}{{_{a_{i}}D_{x_{i}}^{{\alpha}_i}}}

\begin{document}

\title{A formulation of the fractional Noether-type theorem \\
for multidimensional Lagrangians}

\author{Agnieszka B. Malinowska\\
\url{a.malinowska@pb.edu.pl}}

\date{Faculty of Computer Science,
Bia{\l}ystok University of Technology,\\
15-351 Bia\l ystok, Poland}

\maketitle\


\begin{abstract}
This paper presents the Euler--Lagrange equations for fractional variational problems with multiple integrals. The fractional Noether-type theorem for conservative and nonconservative generalized physical systems is proved. Our approach uses well-known notion of the Riemann-Liouville fractional derivative.

\bigskip


\medskip


\medskip

\end{abstract}


\section{Introduction}

Fractional Calculus is the branch of mathematical analysis that studies
derivatives of arbitrary (real or complex) order.
Fractional differentiation and fractional integration
are now recognized as vital mathematical tools to model the behavior
and to understand complex systems, classical or quantum, conservative or
nonconservative, with or without constraints. Several books on the
subject have been written, illustrating the usefulness of the theory in widespread fields of science and engineering (see \cite{TM:rec} for a review).
Some researchers refer to FC as the calculus of the XXI century.
The fractional operators are nonlocal, therefore they are suitable for
constructing models possessing memory effect.

The fractional calculus of variations unifies the calculus of variations and the fractional calculus,
by inserting fractional derivatives into the variational functionals.
This occurs naturally in many problems of physics, mechanics,
and engineering, in order to provide more accurate models
of physical phenomena. The physical reasons for the appearance
of fractional equations are, in general, long-range dissipation
and nonconservatism.
A theory of the fractional calculus of variations started in
1996 with works of Riewe \cite{rie}, in order to better describe nonconservative
systems in mechanics. He explained: ``For conservative systems, variational methods are equivalent to the original used by Newton.
However, while Newton's equations allow nonconservative forces, the later techniques of Lagrangian and Hamiltonian mechanics have no direct way to dealing with them.'' Riewe formulated a Lagrangian with fractional derivatives, and obtained
the respective Euler--Lagrange equation, combining both conservative and
nonconservative cases. Recently, several approaches have been developed to generalize the least action principle and the Euler--Lagrange equations to include fractional derivatives. Investigations cover problems depending
on the Riemann-Liouville fractional derivative (see, e.g., \cite{Baleanu:MUSLIH,Kobelev,Sha}), the Caputo fractional derivative (see, e.g.,
\cite{agrawalCap,MyID:169,withTatiana:Basia}), and others \cite{Almeida:AML,Klimek,MyID:200,Stanislavsky}. We cite here also \cite{El-Nabulsi:Torres07}, where an interesting approach known as the Fractional Action-Like
Variational Approach (FALVA) to model nonconservative dynamical systems is presented, and its extension counterparts \cite{Nabulsi1,El-Nabulsi,Nabulsi2,Nabulsi4}. In \cite{MyID:110,Nabulsi3} the authors study multidimensional FALVA problems, which are generalization of 1D-FALVA problems with fractional derivatives defined in the sense of Cresson. We also mention study of
multidimensional fractional variational problems in context of the coherence problem \cite{Cresson,Cresson2} or via the fractional derivative defined in the sense of Jumarie \cite{MyID:182}. In the present manuscript we rather generalize the Riewe approach to multidimensional fractional variational problems and derive corresponding Euler--Lagrange equations.

Conservative physical systems imply frictionless motion and are a
simplification of the real dynamical world. Almost all processes observed in the physical
world are nonconservative. For
nonconservative dynamical systems the
conservations law are broken so that the standard Lagrangian or
Hamiltonian formalism is no longer valid for describing the behavior
of the system. However, it is still possible to obtain the validity
of Noether's principle using the fractional
calculus of variations. Roughly speaking, one can prove that Noether's conservation laws are
still valid if a new term, involving the nonconservative forces, is
added to the standard constants of motion.
The seminal work \cite{MyID:068} makes use of the notion of Euler--Lagrange
fractional extremal introduced by Riewe,
to prove a Noether-type theorem that
combine conservative and nonconservative cases. After that, other results about the subject appeared \cite{Atanackovic2,Cresson,MyID:094,MyID:149}. Multidimensional version of Noether's theorem is of most relevance to modern physics, particularly to relativistic field theories and to theories of gravity. To the best of our knowledge, the fractional Noether-type theorem for multidimensional Lagrangians is not available. Such a generalization is the aim of this paper. We prove the fractional
Noether theorem for multiple dimensional Lagrangians, and we show how our
results can be use. We trust that this paper will open
several new directions of research and applications.

The paper is organized as follows. In Section~\ref{sec2} we review the necessary
notions of fractional calculus. Our results are given in next sections: in Section~\ref{main}
we derive the Euler--Lagrange equations for fractional variational problems with multiple integrals (Theorem~\ref{th:E-L}), these conditions are then applied to physical problems; in Section~\ref{noether} we prove the fractional Noether-type theorem (Theorem~\ref{theo:tn}).


\section{Fractional Calculus}
\label{sec2}

In this section we review the necessary definitions and facts from
the fractional calculus. For more on the subject we refer the reader to
\cite{kilbas,Podlubny}.

Let $[a,b]$ be a finite interval, $\alpha \in \mathbb{R}$ and $0<\alpha<1$. We begin with Riemann-Liouville Fractional Integrals (RLFI) of order
$\alpha$ of function $f$ which are defined by: the left RLFI
\begin{equation}\label{RLFI1}
{_aI_x^\alpha}f(x)=\frac{1}{\Gamma(\alpha)}\int_a^x
(x-t)^{\alpha-1}f(t)dt,\quad x>a,
\end{equation}
the right RLFI
\begin{equation}\label{RLFI2}
{_xI_b^\alpha}f(x)=\frac{1}{\Gamma(\alpha)}\int_x^b(t-x)^{\alpha-1}
f(t)dt,\quad x<b,
\end{equation}
where $\Gamma(\cdot)$ represents the Gamma function. For $\alpha=0$, we set
${_aI_x^0}f={_xI_b^0}f:=If$, the identity operator. Using the defined fractional integrals
\eqref{RLFI1} and \eqref{RLFI2},
Riemann-Liouville fractional derivatives are defined as: the left
Riemann-Liouville fractional derivative (RLFD)
\begin{equation}\label{RLFD1}
{_aD_x^\alpha}f(x)=\frac{1}{\Gamma(1-\alpha)}\frac{d}{dx}\int_a^x
(x-t)^{-\alpha}f(t)dt=\frac{d}{dx}{_aI_x^{1-\alpha}}f(x),
\end{equation}
the right RLFD
\begin{equation}\label{RLFD2}
{_xD_b^\alpha}f(x)=\frac{-1}{\Gamma(1-\alpha)}\frac{d}{dx}\int_x^b
(t-x)^{-\alpha}
f(t)dt=\left(-\frac{d}{dx}\right){_xI_b^{1-\alpha}}f(x),
\end{equation}
where $\alpha$ is the order of the derivative. The operators \eqref{RLFI1}--\eqref{RLFD2} are obviously linear. Below we present the rules of fractional integration by parts for RLFI and RLFD which are particularly useful for our purposes.
\begin{lem}(\cite{int:partsRef})
\label{ipi}
Let $0<\alpha<1$, $p\geq1$,
$q \geq 1$, and $1/p+1/q\leq1+\alpha$. If $g\in L_p([a,b])$ and
$f\in L_q([a,b])$, then
$\int_{a}^{b}  g(x){_aI_x^\alpha}f(x)dx =\int_a^b f(x){_x I_b^\alpha}
g(x)dx$.
\end{lem}

\begin{lem}\label{byparts}[\textrm{cf.} \cite{rie}]
Let $0<\alpha<1$ and functions $f,g$ obey the assumptions of Lemma~\ref{ipi}.
If $f(a)=f(b)=0$ (or the analogous condition for $g$ is fulfilled), then

\begin{equation}\label{ip}
\int_{a}^{b}  g(x) \, {_aD_x^\alpha}f(x)dx =\int_a^b f(x){_x D_b^\alpha}
g(x)dx.
\end{equation}
\end{lem}

Partial fractional integrals and derivatives are natural generalization of the corresponding onedimensional fractional integrals and derivatives, being taken with respect to one or several variables. For $(x_1,\ldots,x_n)$, $(\alpha_1,\ldots,\alpha_n)$, where $0<\alpha_i<1$, $i=1,\ldots,n$ and $[a_1,b_1]\times \ldots \times [a_n,b_n]$, partial Riemann-Liouville fractional integrals of order $\alpha_k$ with respect to $x_k$ are defined by

\begin{equation*}
{_{a_{k}}I_{x_{k}}^{{\alpha}_k}}f(x_1,\ldots,x_n)
=\frac{1}{\Gamma(\alpha_k)}\int_{a_k}^{x_k}(x_k-t_k)^{\alpha_k-1}f(x_1,\ldots,x_{k-1},t_k,x_{k+1},\ldots,x_n)dt_k,\quad x_k>a_k,
\end{equation*}
\begin{equation*}
{_{x_{k}}I_{b_{k}}^{{\alpha}_k}}f(x_1,\ldots,x_n)
=\frac{1}{\Gamma(\alpha_k)}\int_{x_k}^{b_k}(t_k-x_k)^{\alpha_k-1}f(x_1,\ldots,x_{k-1},t_k,x_{k+1},\ldots,x_n)dt_k,\quad x_k<b_k.
\end{equation*}

Partial Riemann-Liouville derivatives are defined by:
\begin{equation*}
{_{a_{k}}D_{x_{k}}^{{\alpha}_k}}f(x_1,\ldots,x_n)
=\frac{1}{\Gamma(1-\alpha_k)}\frac{\partial}{\partial x_k}\int_{a_k}^{x_k}(x_k-t_k)^{-\alpha_k}f(x_1,\ldots,x_{k-1},t_k,x_{k+1},\ldots,x_n)dt_k,
\end{equation*}
\begin{equation*}
{_{x_{k}}D_{b_{k}}^{{\alpha}_k}}f(x_1,\ldots,x_n)
=-\frac{1}{\Gamma(1-\alpha_k)}\frac{\partial}{\partial x_k}\int_{x_k}^{b_k}(t_k-x_k)^{-\alpha_k}f(x_1,\ldots,x_{k-1},t_k,x_{k+1},\ldots,x_n)dt_k.
\end{equation*}

\section{Fractional Euler--Lagrange Equations}
\label{main}

Consider a physical system characterized by a set of functions
\begin{equation}\label{system}
u_j(t,x_1,\ldots,x_n), \quad j=1,\ldots,m,
\end{equation}
depending on time $t$ and the space coordinates $x_1,\ldots,x_n$. We can simplify the notation by interpreting \eqref{system} as a vector function $u=(u_1,\ldots,u_m)$ and writing $t=x_0$, $x=(x_0,x_1,\ldots,x_n)$, $dx=dx_0dx_1\cdots dx_n$. Then \eqref{system} becomes simply $u(x)$ and is called a vector field. Following \cite{Das,MyID:110}, we introduce the operator $\nabla^{\alpha}$ defined by
$\nabla^{\alpha}=({_{a_0}D_{x_{0}}^{{\alpha}_0}},{_{a_{1}}D_{x_{1}}^{{\alpha}_1}},\cdots, {_{a_{n}}D_{x_{n}}^{{\alpha}_n}})$,
where $\alpha=(\alpha_0,\alpha_1,\ldots,\alpha_n)$, $0<\alpha_i<1$, $i=0,\ldots,n$. Consider the action functional in the form
\begin{equation}\label{problem}
\mathcal{J}(u)=\int_{a_0}^{b_0} L(u,\nabla^{\alpha}u)dx_0
=\int_{\Omega}\mathcal{L}(u,\nabla^{\alpha}u)dx,
\end{equation}
where $R=[a_1,b_1]\times \ldots \times [a_n,b_n]$, $\Omega=R\times [a_0,b_0]$. The boundary conditions are specified by asking that $u(x)=\varphi(x)$ for $x\in \partial \Omega$, where $\varphi:\partial \Omega \rightarrow \mathbb{R}^{m}$ is a given function. The functions $L(u,\nabla^{\alpha}u)$ and $\mathcal{L}(u,\nabla^{\alpha}u)$ are called the Lagrangian and Lagrangian density of the field, respectively. The advantage of using $\nabla^{\alpha}$ in the formulation of problem \eqref{problem} lies in the fact that fractional derivatives take into account the space history and time history. Therefore, comparing with integer order derivatives, they are more appropriate for representation of natural phenomena. We assume that: (i) The set of admissible function consists of all function $u: \Omega \rightarrow \mathbb{R}^m$, such that $\rl u_j\in C(\Omega,\mathbb{R})$,  $i=0,\ldots,n$, $j=1,\ldots, m$, and $u|_{\partial \Omega}=\varphi$; (ii) $\mathcal{L}\in C^1(\mathbb{R}^m\times\mathbb{R}^{m(n+1)}; \mathbb{R})$; (iii) ${_{x_i}D_{b_{i}}^{{\alpha}_i}}\frac{\partial \mathcal{L}}{\partial\rl u_j}$,  $i=0,\ldots,n$, $j=1,\ldots,m$ are continuous on $(a_0,b_0)\times \ldots \times (a_n,b_n)$ and integrable on $\Omega$ for all admissible functions. Applying the principle of stationary action to \eqref{problem} we obtain the multidimensional fractional Euler--Lagrange equations for the field.

\begin{thm}\label{th:E-L}
If function $u$ minimizes the action functional \eqref{problem}, then $u$ satisfies the
multidimensional fractional Euler--Lagrange differential equations:
\begin{equation}\label{E-L}
\frac{\partial \mathcal{L}}{\partial u_j}+\sum_{i=0}^n {_{x_i}D_{b_{i}}^{{\alpha}_i}}\frac{\partial \mathcal{L}}{\partial\rl u_j}=0,\quad j=1,\ldots,m.
\end{equation}
\end{thm}
\begin{proof}
The variations for this problem are functions $h:\Omega\rightarrow \mathbb{R}^{m}$ such that $\rl h_j\in C(\Omega,\mathbb{R})$,  $i=0,\ldots,n$, $j=1,\ldots, m$, and $h$  vanishes on the boundary: $h(0)=0$, $x\in\partial\Omega$. The condition that a function $u:\Omega\rightarrow \mathbb{R}^{m}$ be a critical point of the action $\mathcal{J}(u)$ is $\left. \frac{d}{d \varepsilon}\mathcal{J}(u+\varepsilon h)\right|_{\varepsilon =0}=0$
for all variations $h$ with $h|_{\partial \Omega}=0$. Differentiating under the integral sign we obtain
\begin{equation*}
0=\sum_{j=1}^{m}\int_{\Omega}\left(\frac{\partial \mathcal{L}}{\partial u_j}h_j+\sum_{i=0}^n \frac{\partial \mathcal{L}}{\partial\rl u_j}\rl h_j \right ) dx
\end{equation*}
The Fubini theorem allows us to rewrite integrals as the iterated integrals so that we can use the integration by parts formula \eqref{ip}:
\begin{equation*}
\int_{\Omega}\left(\frac{\partial \mathcal{L}}{\partial u_j}+\sum_{i=0}^n {_{x_i}D_{b_{i}}^{{\alpha}_i}}\frac{\partial \mathcal{L}}{\partial\rl u_j}\right) h_j dx, \quad j=1,\ldots m
\end{equation*}
since $h|_{\partial \Omega}=0$. Therefore,
\begin{equation*}
0=\sum_{j=1}^{m}\int_{\Omega}\left(\frac{\partial \mathcal{L}}{\partial u_j}+\sum_{i=0}^n {_{x_i}D_{b_{i}}^{{\alpha}_i}}\frac{\partial \mathcal{L}}{\partial\rl u_j}\right) h_j dx
\end{equation*}
By fundamental lemma of the calculus of variations (see Section~2.2 in \cite{G:H}), $u$ satisfies \eqref{E-L}.
\end{proof}

In the following we present physical examples. The main aim is to illustrate what are the Euler--Lagrange equations when we have a Lagrangian density depending on fractional derivatives.

\begin{ex}
Consider a minimizer of
$\mathcal{J}(u)=\frac{1}{2}\int \int_{R}\sum _{i=1}^{2}\left(\left(\rl u\right)^2-fu\right)dx,$
where $f:R\rightarrow \mathbb{R}$ is a given function, in a set of functions that satisfy condition $u=\varphi$ on $\partial R$, where $\varphi$ is a given function defined on the boundary $\partial R$ of $R=[a_1,b_1]\times [a_2,b_2]$. By Theorem~\ref{th:E-L} a minimizer satisfies the following equation $\sum _{i=1}^{2} {_{x_i}D_{b_{i}}^{{\alpha}_i}} \rl u=f.$
Observe that if $\alpha_i$ goes to $1$, then the operator $\rl$, $i=1,2$, can be replaced with $\frac{\partial}{\partial x_i}$, and the operator ${_{x_i}D_{b_{i}}^{{\alpha}_i}} $, $i=1,2$, can be replaced with $-\frac{\partial}{\partial x_i}$ (see \cite{Podlubny}). Thus, for $\alpha\rightarrow 1$ we obtain the Poisson equation that arises in the potential theory in electrostatic fields.
\end{ex}

\begin{ex}\label{co:sys}
Consider a motion of medium whose displacement may be described by a scalar function $u(t,x)$, where $x=(x_1,x_2)$. For example, this function might represent the transverse displacement of the membrane. Suppose that the kinetic energy $T$ and potential energy $V$ of the medium are given by:
$T\left(\frac{\partial u}{\partial t}\right)=\frac{1}{2}\int \rho \left(\frac{\partial u}{\partial t}\right)^2 dx$, $V(u)=\frac{1}{2}\int k \|\nabla u \|^2dx$, where $\rho (x)$ is a mass density and $k(x)$ is a stiffness, both assume positive. Then, the classical action functional is
$ \mathcal{J}(u)=\frac{1}{2}\int \int_{\Omega}\left(\rho \left(\frac{\partial u}{\partial t}\right)^2-k \|\nabla u\|^2 \right)dx dt.$
When we have the Lagrangian with the kinetic term depending on a fractional derivative, then
the fractional action functional has the form

\begin{equation}\label{ex:3}
\mathcal{J}(u)=\frac{1}{2}\int \int_{\Omega}\left(\rho \left({_{0}D_{t}^{{\alpha_0}}}u\right)^2-k \|\nabla u\|^2 \right)dx dt.
\end{equation}
The fractional Euler--Lagrange equation satisfied by a stationary point of \eqref{ex:3} is
\begin{equation*}
\rho {_{t}D_{t_1}^{{\alpha_0}}} {_{0}D_{t}^{{\alpha_0}}}u-\nabla (k \nabla u)=0.
\end{equation*}
If $\rho$ and $k$ are constants, then equation ${_{t}D_{t_1}^{{\alpha}}} {_{0}D_{t}^{{\alpha}}}u-c^2\Delta u=0$ where $c^2=k/\rho$ can be called the time-fractional wave equation. In the case when the kinetic and potential energy depend on fractional derivatives, the action functional for the system has the form
\begin{equation}\label{ex:2}
\mathcal{J}(u)=\frac{1}{2}\int \int_{\Omega}\left[\rho \left({_{0}D_{t}^{{\alpha_0}}}u\right)^2-k(({_{a_1}D_{x_1}^{{\alpha_1}}}u)^2+({_{a_2}D_{x_2}^{{\alpha_2}}} u)^2) \right]dx dt.
\end{equation}
The fractional Euler--Lagrange equation satisfied by a stationary point of \eqref{ex:2} is
\begin{equation*}
\rho {_{t}D_{t_1}^{{\alpha_0}}} {_{0}D_{t}^{{\alpha_0}}}u-k({_{x_1}D_{b_1}^{{\alpha_1}}}{_{a_1}D_{x_1}^{{\alpha_1}}} u+{_{x_2}D_{b_2}^{{\alpha_2}}} {_{a_2}D_{x_2}^{{\alpha_2}}} u)=0.
\end{equation*}
If $\rho$ and $k$ are constants, then
${_{t}D_{t_1}^{{\alpha}}} {_{0}D_{t}^{{\alpha}}}u-c^2({_{x_1}D_{b_1}^{{\alpha_1}}}{_{a_1}D_{x_1}^{{\alpha_1}}} u+{_{x_2}D_{b_2}^{{\alpha_2}}} {_{a_2}D_{x_2}^{{\alpha_2}}} u)=0$
can be called the space-and time-fractional wave equation. Observe that in the limit, $\alpha,\alpha_1,\alpha_2\rightarrow 1$, we obtain
the classical wave equation
$\frac{\partial^2 u}{\partial t^2}-c^2\Delta u=0$ with wave-speed $c$.
\end{ex}

\begin{rem}
Recently, several authors have investigated the
fractional diffusion/wave equation and its special properties.
Mostly fractional wave equations are obtained by changing classical partial derivatives by fractional, it can be the Riemann-Liouville, Caputo or another one \cite{Blackledge,Modes,Momani,Yuste,Schneider}. Here, following \cite{Cresson,Cresson2,Nabulsi3,Parsian} we use the different approach. We start with a variational formulation of a physical process in which one modifies the Lagrangian density by replacing integer order derivatives with fractional ones. Then the action integral in the sense of Hamilton is minimized and the governing equation of a physical process is obtained. This approach can be explain via the notion of fractional embedding, we refer the reader to \cite{Cresson,Cresson2}. The main remark is that equations by themselves do not have a universal significance, their form depending mostly on the coordinates systems being used to derive them. On the contrary the underlying first principle like the least-action principle carry an information which is of physical interest and not related to the coordinates system which is used.
\end{rem}
\begin{rem}
It should be mentioned, that since fractional operators are nonlocal, it can be extremely challenging to find analytical
solutions to fractional problems of the calculus of variations and, in many cases, solutions may not exist.
The question of existence of solutions to fractional variational problems is a complete open area of research. This needs attention.
As mentioned by Young,
the calculus of variations has born from the study
of necessary optimality conditions, but any such theory is ``naive'' until
the existence of minimizers is verified. The process leading
to the existence theorems was introduced by
Leonida Tonelli in 1915 by the so-called direct method.
During two centuries, mathematicians
were developing ``the naive approach to the calculus of variations''.
Similar situation happens now with the fractional
calculus of variations: the subject is only fifteen years old,
and is still in the ``naive period''. We believe
time has come to address the existence question,
and this will be considered in a forthcoming paper.
\end{rem}

\section{Fractional Noether-type Theorem}
\label{noether}

In this section we prove the fractional Noether-type theorem. We start by
introducing the notion of variational invariance.

\begin{defn}
\label{def:inv1}
Functional \eqref{problem} is said to be invariant under an
$\varepsilon$-parameter group of infinitesimal transformations
$\bar{u}(x)= u + \varepsilon\xi(x,u) + o(\varepsilon)$ if
\begin{equation}\label{inv}
\int_{\Omega^*}\mathcal{L}(u,\nabla^{\alpha}u)dx
=\int_{\Omega^*}\mathcal{L}(\bar{u},\nabla^{\alpha}\bar{u})dx
\end{equation}
for any $\Omega^*\subseteq \Omega$.
\end{defn}

The next theorem establishes a necessary condition of invariance.

\begin{thm}
If functional \eqref{problem} is invariant in the sense of
Definition~\ref{def:inv1}, then
\begin{equation}
\label{nes:noe}
\sum_{j=1}^{m}\frac{\partial \mathcal{L}}{\partial u_j}\xi_j+\sum_{j=1}^{m}\sum_{i=0}^n \frac{\partial \mathcal{L}}{\partial\rl u_j}\rl \xi_j=0.
\end{equation}
\end{thm}

\begin{proof}
Equation \eqref{inv} is equivalent to
$\mathcal{L}(u,\nabla^{\alpha}u)
=\mathcal{L}( u + \varepsilon\xi(x,u) + o(\varepsilon),\nabla^{\alpha}( u + \varepsilon\xi(x,u) + o(\varepsilon)))$.
Differentiating both sides of with respect to $\varepsilon$, then substituting
$\varepsilon=0$, and using the definitions and properties of the
fractional partial derivatives, we obtain equation \eqref{nes:noe}.
\end{proof}

The following definition is similar to \cite{MyID:068} and useful in order to formulate the fractional Noether-type theorem.

\begin{defn}
\label{def:oprl}
For $0<\gamma <1$ and a given ordered pair of functions $f$ and $g$ ($f,g:\mathbb{R}^{n+1}\rightarrow \mathbb{R}$) we introduce the following operator:
$\mathcal{D}^{\gamma}\left(f,g\right) = f \cdot {_{a_i}D_{x_i}^\gamma} g - g \cdot {_{x_i}D_{b_i}^\gamma}f$,
where $x_i \in [a_i,b_i]$, $i=0,\ldots ,n$.
\end{defn}

We now prove the fractional Noether-type theorem.

\begin{thm}
\label{theo:tn}
If functional \eqref{problem} is invariant in the sense of
Definition~\ref{def:inv1}, then
\begin{equation}\label{eq:LC}
\sum_{i=0}^n\sum_{j=1}^{m}\mathcal{D}^{\alpha_i}\left(\frac{\partial \mathcal{L}}{\partial\rl u_j},\xi_j\right)=0
\end{equation}
along any solution of \eqref{E-L}.
\end{thm}

\begin{proof}
Using the fractional Euler--Lagrange equations \eqref{E-L},
we have:
\begin{equation}
\label{eq:eldf3}
\sum_{j=1}^{m}\frac{\partial \mathcal{L}}{\partial u_j}=-\sum_{j=1}^{m}\sum_{i=0}^n {_{x_i}D_{b_{i}}^{{\alpha}_i}}\frac{\partial \mathcal{L}}{\partial\rl u_j}.
\end{equation}
Substituting \eqref{eq:eldf3} in the necessary condition of
invariance \eqref{nes:noe} we get
\begin{multline*}
0=-\sum_{j=1}^{m}\sum_{i=0}^n {_{x_i}D_{b_{i}}^{{\alpha}_i}}\frac{\partial \mathcal{L}}{\partial\rl u_j}\xi_j+\sum_{j=1}^{m}\sum_{i=0}^n \frac{\partial \mathcal{L}}{\partial\rl u_j}\rl \xi_j\\
=\sum_{i=0}^n \left(\sum_{j=1}^{m} \frac{\partial \mathcal{L}}{\partial\rl u_j}\rl \xi_j-\sum_{j=1}^{m} {_{x_i}D_{b_{i}}^{{\alpha}_i}}\frac{\partial \mathcal{L}}{\partial\rl u_j}\xi_j\right).
\end{multline*}
By definition of operator $\mathcal{D}^{\gamma}$ we obtain equation \eqref{eq:LC}.
\end{proof}

\begin{rem}
In the particular case when $\alpha_i$ goes to $1$,  $i=0,\ldots,n$, we get
from the fractional conservation law \eqref{eq:LC} the classical Noether conservation law
without transformation of the independent variables (that is, conservation of current):
\begin{equation}\label{cl}
\sum_{i=0}^n \frac{\partial}{\partial x_i}\sum_{j=1}^{m}\frac{\partial \mathcal{L}}{\partial \left(\frac{\partial u_j}{\partial x_i}\right)}\xi_j=0
\end{equation}
along any extremal surface of
$\mathcal{J}(u)=\int_{\Omega}\mathcal{L}(u,\nabla u)dx.$
\end{rem}

\begin{ex}
Let us consider again the Lagrangian given in Example~\ref{co:sys}: $L=T-V$. Then the Hamiltonian, $H=T+V$, is conserved (see \cite{gold}). However, if the kinetic and potential energy depend on fractional derivatives, that is the action functional has the form \eqref{ex:2}, then it can be showed that the Hamiltonian is not conserved (\textrm{cf.} \cite{Klimek,rie}). Observe that functional \eqref{ex:2} is invariant (in the sense of Definition~\ref{def:inv1}) under transformations $\bar{u}(t,x)=u(t,x)+\varepsilon t^{\alpha_0-1}(x_1-a_1)^{\alpha_1-1}(x_2-a_2)^{\alpha_2-1}$. Therefore, by Theorem~\ref{theo:tn}, we obtain:
\begin{equation}\label{ap}
\sum_{i=0}^2\mathcal{D}^{\alpha_i}\left(\frac{\partial \mathcal{L}}{\partial\rl u},x_0^{\alpha_0-1}(x_1-a_1)^{\alpha_1-1}(x_2-a_2)^{\alpha_2-1}\right)=0,
\end{equation}
 with $x_0=t$ and $a_0=0$, along any extremal surface of \eqref{ex:2}. Note that when $\alpha_i$ goes to $1$,  $i=0,1,2$, expression \eqref{ap} has the form \eqref{cl}.
This shows that the proved fractional Noether-type theorem can be applied for conservative and nonconservative physical systems.
\end{ex}


\section*{Acknowledgments}
The author is supported
by Bia{\l}ystok University of Technology Grant S/WI/2/2011. I would like to thank the anonymous referees for their useful comments and valuable suggestions.



\end{document}